\documentclass[pdflatex,sn-mathphys-num]{sn-jnl}

\usepackage{graphicx}%
\usepackage{multirow}%
\usepackage{amsmath,amssymb,amsfonts}%
\usepackage{amsthm}%
\usepackage{mathrsfs}%
\usepackage[title]{appendix}%
\usepackage{xcolor}%
\usepackage{textcomp}%
\usepackage{manyfoot}%
\usepackage{booktabs}%
\usepackage{algorithm}%
\usepackage{algorithmic}%
\usepackage{listings}%

\usepackage{lipsum}
\usepackage{graphicx}
\usepackage{subfigure} 
\usepackage{caption}
\usepackage{hyperref}
\usepackage{mathtools}
\usepackage{epstopdf}

\theoremstyle{thmstyleone}%
\newtheorem{theorem}{Theorem}
\newtheorem{lemma}{Lemma}
\newtheorem{corollary}{Corollary}
\theoremstyle{thmstyletwo}%
\newtheorem{remark}{Remark}%

\theoremstyle{thmstylethree}%
\newtheorem{definition}{Definition}%

\newcommand{\argmin}{\operatornamewithlimits{\arg\,\min}}

\begin{document}

\title[MPFA: A Pareto Front Approximation Method for  Riemannian  Bi-objective Optimization]{MPFA: A Pareto Front Approximation Method for  Riemannian  Bi-objective Optimization}


\author[1,2]{\fnm{Kaiping} \sur{Liu}}\email{kaipliu@163.com}

\author*[2]{\fnm{P.-A.} \sur{Absil}}\email{pa.absil@uclouvain.be}

\author[1]{\fnm{Jiawei} \sur{Chen}}\email{j.w.chen713@163.com}

\affil[1]{\orgdiv{School of Mathematics and Statistics}, \orgname{Southwest University}, \orgaddress{\city{Chongqing}, \postcode{400715}, \country{China}}}

\affil[2]{\orgdiv{ICTEAM Institute}, \orgname{UCLouvain}, \orgaddress{\city{1348 Louvain-la-Neuve}, \country{Belgium}}}


\abstract{We propose a Pareto front approximation (MPFA) method for smooth bi-objective optimization problems on Riemannian manifolds based on a Hermite interpolation technique. Compared with the existing multiobjective optimization numerical algorithms, the proposed method can generate a continuous approximate Pareto front without multiple initial points. We establish convergence of the proposed method and analyze the approximation error of the resulting Pareto front. Numerical experiments on several test problems demonstrate that the proposed approach can effectively approximate the Pareto front with high accuracy and reasonable computational cost. Furthermore, the method is applied to a bi-objective formulation of sparse principal component analysis, illustrating its practical applicability in data analysis problems.}





\maketitle

\section{Introduction}\label{sec1}

Multiobjective optimization (MOO), also referred to as multicriteria optimization, concerns the simultaneous optimization of several objective functions. Such problems arise naturally in many applications including engineering design, operations management, economics, finance, and agriculture \cite{w90,m99,g18}, where multiple conflicting criteria must be considered in the decision-making process. In contrast to single-objective optimization, in general there exists no single solution that simultaneously minimizes all objective functions. Instead, the concept of optimality is replaced by the notion of Pareto optimality \cite{m99, sn85}.  The set of all Pareto optimal points forms the Pareto set, and its image in the objective space is called the Pareto front. The approximation of the Pareto front is of central importance since it reveals the trade-offs among competing objectives and provides valuable information for decision makers.

In recent years, multiobjective optimization problems (MOPs) on Riemannian manifolds have received significant attention.  Some applications of multiobjective optimization on Riemannian
manifolds can be found in \cite{cs21}.
Consequently, the development of efficient algorithms for MOPs on manifolds has become an active research area.  Numerous Riemannian optimization methods have been proposed, including steepest descent methods \cite{bf12,bn13}, subgradient method \cite{bn13s}, proximal point method \cite{bn18}, conjugate gradient methods \cite{nh23}, and trust-region methods \cite{en23}. However, these methods are designed to compute individual Pareto optimal points rather than directly generating the entire Pareto front. 
To construct the  entire Pareto front, one  must execute the optimization algorithm many times over by using different starting points, which can be computationally expensive. Can the entire Pareto front be approximated more efficiently, perhaps from a single initial point? This question defines a crucial research gap, as methodologies with such a capability remain largely unexplored.

Consider the following bi-objective optimization problem on a Riemannian manifold $\mathcal{M}$,
\begin{equation}\label{bo}
\min_{x\in \mathcal{M} } F(x) \coloneqq ( f_1(x), f_2(x) ),
\end{equation}
where $f_1,f_2:\mathcal{M}\to\mathbb{R}$ are continuously differentiable functions. A classical strategy for handling multiobjective optimization problems is scalarization, which converts the multiobjective problem into the single-objective problems. In particular, it transforms \eqref{bo} into the single-objective optimization problem:
\begin{equation}\label{sp}
\min_{x\in \mathcal{M} } \varphi_\mu (x) \coloneqq f_1(x)+\mu f_2(x), \quad \mu>0.
\end{equation}
For each $\mu>0$, 
a solution of the problem \eqref{sp} on the compact Riemannian manifold $\mathcal{M}$ is denoted by
\begin{equation}\label{fgamma}
\gamma(\mu)\coloneqq \argmin_{x\in\mathcal{M}} \varphi_\mu (x).
\end{equation}
Consequently, the set
$
\{(f_1(\gamma(\mu)),f_2(\gamma(\mu))) : \mu>0\}
$ 
provides a subset of the Pareto front.

However, obtaining a sufficiently dense sampling of the Pareto front by directly solving the scalarized problem for many values of the parameter $\mu$ can be computationally expensive, since each evaluation requires solving a nonlinear optimization problem on the manifold. This issue becomes particularly pronounced when high accuracy is required or when the optimization problem itself is costly to solve.

In this paper, to overcome this difficulty, we propose a manifold Pareto front approximation (MPFA) framework based on curve fitting techniques on manifolds. The main idea is to compute a small number of sample solutions $\gamma(\mu_1),\ldots,\gamma(\mu_N)$ of the scalarized problem \eqref{sp} using the Riemannian steepest descent method and the Riemannian Newton method. These solutions serve as anchor points for constructing an approximation of the solution curve $\gamma$. By estimating tangent information along the curve and employing Hermite interpolation on the manifold \cite{z20}, we obtain a smooth approximation $\tilde{\gamma}$ of $\gamma$. Once this approximation is available, points on an approximate Pareto front can be generated efficiently by evaluating $(f_1(\tilde{\gamma}(\mu)),f_2(\tilde{\gamma}(\mu)))$ for different values of $\mu$, without repeatedly solving expensive optimization problems. We provide the error and convergence analysis for the proposed algorithm, which can be found in Section \ref{sec4}. 
Numerical experiments for several test problems on the sphere manifold show that MPFA produces accurate Pareto front approximations using only a single initial point. In particular, the method based on Hermite interpolation  achieves the smallest approximation errors compared with piecewise geodesic and Lagrange interpolation.
Moreover, the approximation error decreases steadily as the number of anchor points increases, and the observed convergence rate is close to the theoretical prediction. Finally, the proposed framework is successfully applied to a bi-objective sparse PCA problem.

The main contributions of this work are summarized as follows.
\begin{itemize}
\item We propose a  Pareto front approximation method for
bi-objective optimization problems on Riemannian manifolds, based on
scalarization and manifold Hermite interpolation.

\item We derive a tangent characterization of the scalarized solution
path using the implicit function theorem, which enables the construction
of an interpolation scheme on the manifold.

\item We establish an approximation error estimate for the proposed
method. Moreover, numerical experiments demonstrate  the effectiveness of the proposed algorithm.

\end{itemize}

Whereas Pareto front interpolation methods are available in the Euclidean case~\cite{hm12, bs17, hm19}, it seems that our MPFA framework remains novel when the manifold $\mathcal{M}$ reduces to a Euclidean space.

The remainder of this paper is organized as follows. Section \ref{sec2} briefly reviews necessary concepts from Riemannian geometry and multiobjective optimization. Section \ref{sec3} presents the MPFA method. Section \ref{sec4} establishes the error and convergence results. Numerical experiments are reported in Section \ref{sec5}. We make some final conclusions in Section \ref{sec6}.

\section{Preliminaries}\label{sec2}

In this section, we briefly review some basic concepts and notation from Riemannian geometry and multiobjective optimization that will be used throughout the paper. These concepts provide the necessary  foundation for the development and analysis of the proposed method.

The Riemannian concepts used in this paper follow standard references on
Riemannian optimization; see, for example, \cite{am08, b23}.
A Riemannian manifold $\mathcal{M}$ is a smooth manifold endowed with a
Riemannian metric $\langle \cdot , \cdot \rangle_x$ for each $x \in \mathcal{M}$.
The metric defines an inner product on the tangent space $T_x\mathcal{M}$ and
induces the norm
$\|\eta_x\|_x = \sqrt{\langle \eta_x,\eta_x\rangle_x}, \, \eta_x \in T_x\mathcal{M}$.
When the point $x$ is clear from the context, we simply write $\|\eta\|$. Let $x, y \in \mathcal{M}$, $\mathrm{dist}(x,y)$ denotes the Riemannian distance between $x$ and $y$.

The Riemannian exponential map at $x$ is
denoted by $\mathrm{Exp}_x : T_x\mathcal{M} \rightarrow \mathcal{M}$,
which maps a tangent vector $\eta_x$ to the point reached at time $1$ by the geodesic
starting from $x$ with initial velocity $\eta_x$.
Following \cite[Definition~4.3]{s96}, let $\mathcal I_x$
denote the interior set of $x$, and let $\mathcal I_x^T$
denote the corresponding tangent interior set at $x$.
For each $y\in\mathcal I_x$, the logarithm map is defined by $\mathrm{Log}_x(y):=\xi_x$,
where $\xi_x\in\mathcal I_x^T$ is the unique tangent vector satisfying $\mathrm{Exp}_x(\xi_x)=y$.
For a differentiable mapping $f : \mathcal{M} \rightarrow \mathcal{N}$,
its differential at $x$ is the linear mapping
$\mathrm{d}f(x) : T_x\mathcal{M} \rightarrow T_{f(x)}\mathcal{N}$.
In particular, the chain rule extends naturally to the manifold setting.

Next, we introduce some concepts of the multiobjective optimization.
Let $\mathbb{R}_{++}^n$ and $\mathbb{R}_+^n$ be the positive and  nonnegative orthant of $\mathbb{R}^n$, respectively. For $x,y \in \mathbb{R}^n$, the partial
order $x \preceq y$ holds if and only if $y-x \in \mathbb{R}_+^n$. If $n=1$, then $\preceq$  is equivalent to $\leq$. Likewise, the partial
order $x \prec y$ holds if and only if $y-x \in \mathbb{R}_{++}^n$. If $n=1$, then $\prec$  is equivalent to $<$. Now, let us recall several useful definitions and results.
\begin{definition} \cite{bf12, bn13, e05, g68}
	\item {\bf (i)} A point \(\hat{x} \in \mathcal{M}\)  is called \emph{efficient} (or \emph{Pareto optimal}), if there exists no $x \in \mathcal{M}  $ such that
	\begin{equation*}
		F(x ) \preceq F(\hat{x} )~{\rm and} ~F(x ) \neq F(\hat{x} ).
	\end{equation*}
	
	\item {\bf (ii)}  A point \(\hat{x} \in \mathcal{M} \) is called \emph{properly efficient}, if it is efficient and there is a real number \(T > 0\) such that for all \(i\) and \(x \in \mathcal{M} \) satisfying \(f_i(x) < f_i(\hat{x})\), there exists an index \(j\) such that \(f_j(\hat{x}) < f_j(x)\) such that
	
	\[\frac{f_i(\hat{x}) - f_i(x)}{f_j(x) - f_j(\hat{x})} \leq T.
	\]
\end{definition}

The next result is a direct application of the extreme value theorem \cite{rw98}.
\begin{lemma}
	Let  $\mathcal{M}$ be a compact Riemannian manifold and $f_1, f_2: \mathcal{M} \rightarrow \mathbb{R}$ be continuous functions. Then, for all $\mu>0$, \eqref{sp} admits at least one global minimizer $x_\mu \in \mathcal{M}$.
\end{lemma}

\begin{lemma}\label{sp_pes}
	Let $\mu_k>0, k=1,\dots,p$, be positive weights and $X$ be a feasible set. If $\bar{x}$ is an optimal solution of  
	\begin{equation*}
		\min_{x\in X } \sum_{k=1}^{p} \mu_k f_k(x),
	\end{equation*}
	then $\bar{x}$ is a properly efficient solution of 
	\begin{equation*}
		\min_{x\in X } \left( f_1(x),\dots, f_p(x) \right).
	\end{equation*}
\end{lemma}
\begin{proof}
	This follows directly from \cite[Theorem 1]{g68} by considering the equivalent vector maximization problem with objective functions $-f_k, k=1, \ldots, p$.
\end{proof}

\begin{remark}\label{rem:convexity}
	It is worth noting that while Lemma \ref{sp_pes} ensures that every solution $\gamma(\mu)$ of the scalarized problem \eqref{sp} is properly efficient, the converse proposition requires additional assumptions, such as the (geodesic) convexity of the objective functions and the manifold $\mathcal{M}$. However, in this paper, we focus on efficiently approximating the curve $\gamma$.
\end{remark}

\section{MPFA Method}\label{sec3}

In this section, we present the manifold Pareto front approximation (MPFA) algorithm (Algorithm \ref{alg:mpfa}) for solving the bi-objective optimization problem on Riemannian manifolds \eqref{bo}. 
The algorithm first computes a set of anchor points by solving a family of scalarized optimization problems using the steepest descent solver provided by \texttt{Manopt} \cite{Manopt2014}. More precisely, for each fixed scalarization parameter $\mu$, we
first apply  \texttt{steepestdescent} solver to obtain an initial
approximation of a stationary point of \eqref{sp}.
This initial approximation is then refined by a Riemannian Newton-type
correction. The use of the \texttt{steepestdescent} provides a globally
convergent initialization, while the Newton method (Algorithm \ref{alg:newton}) improves the local
accuracy of the anchor points used in the subsequent Hermite interpolation.
Based on these anchor points and the estimated velocity information along the solution path, a smooth approximation of the solution curve is constructed via Hermite interpolation on the manifold. Finally, we propose the approximate Pareto front.

First, the algorithm of producing an approximation to the solution curve $\gamma$ for the scalarized problem \eqref{sp} is given in Algorithm~\ref{alg:mpfa}.

\begin{algorithm}[htbp]
	\caption{Manifold Pareto Front Approximation (MPFA)}
	\label{alg:mpfa}
	\begin{algorithmic}[1]
		\STATE \textbf{Given:}
		Manifold $\mathcal{M}$, objective functions $f_1,f_2:\mathcal{M}\to\mathbb{R}$,
		Riemannian gradients $\mathrm{grad}\,f_1$, $\mathrm{grad}\,f_2$,
		$N \in\mathbb{N}$, 
		tolerance $\epsilon>0$,
		initial point $x_0\in \mathcal{M}$

		\STATE Determine $\mu_{\min},\mu_{\max}$ adaptively 
		using Algorithm~\ref{alg:adaptive-mu}
		
		\STATE $\text{AnchorPoints} \coloneqq \emptyset$, 
		$\text{MuList} \coloneqq \emptyset$
		
		\FOR{$i=1$ to $N$}
		\STATE $\mu_i\coloneqq 
		\mu_{\min} + (\mu_{\max}-\mu_{\min})\frac{i-1}{N-1}$
		\IF{$i=1$} 
		\STATE Compute $$x_1 \coloneqq
		\argmin_{x\in\mathcal M}\varphi_{\mu_1}(x),$$
		using a Riemannian steepest descent method (see \texttt{Manopt} \cite{Manopt2014}) 
		\ELSE
		\STATE Compute $$x_i \coloneqq
		\argmin_{x\in\mathcal M}\varphi_{\mu_{i}}(x),$$
		using a Riemannian Newton method  (Algorithm~\ref{alg:newton}) initialized at $x_{i-1}$
		\ENDIF 
		\STATE $\textrm{AnchorPoints}(i) \coloneqq x_i$
		\STATE $\textrm{MuList}(i) \coloneqq \mu_i$
		
		\STATE  $v_i \coloneqq -{\operatorname{Hess}  \varphi_{\mu_i}(x_i)}^{-1}\operatorname{grad} f_2(x_i)$
		\STATE $V[i] \coloneqq v_i$
		\ENDFOR
		
		\STATE $\tilde{\gamma} \coloneqq
		\textsc{HermiteInterpolationOnManifold}
		(\text{MuList}, \text{AnchorPoints}, V,\mathcal{M})$  (see Algorithm~\ref{alg:hermite-interp}) \label{11sol}
		\STATE Define $\Phi:[\mu_{\min},\mu_{\max}]\to\mathbb{R}^2$ by
		\[
		\Phi(\mu)=\big(f_1(\tilde{\gamma}( \mu)),
		f_2(\tilde{\gamma}(\mu))\big).
		\]
	\end{algorithmic}
\end{algorithm}

\begin{algorithm}
\caption{Adaptive Determination of $\mu_{\min}$ and $\mu_{\max}$}
\label{alg:adaptive-mu}
\begin{algorithmic}[1]
	\STATE Given: Manifold $\mathcal{M}$, functions $f_1, f_2$, small constants $\delta=10^{-6}$, $\epsilon_1=10^{-12}$, $c>1$, $w>1$, $K\in \mathbb{N}$. 
	
	\STATE Compute $$x_1^* = \argmin_{x \in \mathcal{M}} f_1(x)$$ using a Riemannian steepest descent method (see \texttt{Manopt}) 
	\STATE Compute $$x_2^* =\argmin_{x \in \mathcal{M}} f_2(x)$$ using a Riemannian steepest descent method (see \texttt{Manopt})

	\STATE Set
	\[
	\mu_{\mathrm{low}}
	=
	\frac{
		\|\operatorname{grad}f_1(x_1^*)\|+\delta
	}{
		\|\operatorname{grad}f_2(x_1^*)\|+\delta
	},
	\qquad
	\mu_{\mathrm{high}}
	=
	\frac{
		\|\operatorname{grad}f_1(x_2^*)\|+\delta
	}{
		\|\operatorname{grad}f_2(x_2^*)\|+\delta
	}.
	\]

	\STATE $\mu_{\text{test}} = \mu_{\text{low}} / c$
	\FOR{$k=1,\ldots,K$}
	\STATE Compute
	$$
	x_{\mathrm{test}}
	=
	\argmin_{x\in\mathcal M}
	\bigl(f_1(x)+\mu_{\mathrm{test}}f_2(x)\bigr),
	$$
	using a Riemannian steepest descent method (see \texttt{Manopt}) initialized at $x_1^*$.
	\IF{$\text{dist}_{\mathcal{M}}(x_{\text{test}}, x_2^*) > \epsilon_1$}
	\STATE $\mu_{\text{low}} = \mu_{\text{test}}$
	\STATE $\mu_{\text{test}} = \mu_{\text{test}} / c$ \COMMENT{Test smaller $\mu$}
	\ELSE
	\STATE Break.
	\ENDIF
	\ENDFOR
	
	\STATE $\mu_{\text{test}} = \mu_{\text{high}} \times w$
	\FOR{$k=1,\ldots,K$}
	\STATE Compute
	$$
	x_{\mathrm{test}}
	=
	\argmin_{x\in\mathcal M}
	\bigl(f_1(x)+\mu_{\mathrm{test}}f_2(x)\bigr),
	$$ 
	using a Riemannian steepest descent method (see \texttt{Manopt}) 
	initialized at $x_2^*$.
	\IF{$\text{dist}_{\mathcal{M}}(x_{\text{test}}, x_1^*) > \epsilon_1$}
	\STATE $\mu_{\text{high}} = \mu_{\text{test}}$
	\STATE $\mu_{\text{test}} = \mu_{\text{test}} \times w$ \COMMENT{Test larger $\mu$}
	\ELSE
	\STATE Break.
	\ENDIF
	\ENDFOR
	
	\STATE Set $$\mu_{\min} = \mu_{\text{low}}, \quad
	\mu_{\max} = \mu_{\text{high}}$$ 
	
	\RETURN $\mu_{\min}, \mu_{\max}$
\end{algorithmic}
\end{algorithm}

\begin{remark} 
Algorithm \ref{alg:adaptive-mu} is introduced to adaptively determine the range of scalarization weights $[\mu_{\min}, \mu_{\max}]$ that 
provides a practical interval for sampling parameters.
In bi-objective optimization, a fixed weight interval carries the risk of either omitting significant segments of the Pareto front when too narrow, or incurring unnecessary computational expense on redundant solutions when too wide. By exploring the extremes of the scalarized problem,  Algorithm \ref{alg:adaptive-mu} estimates $\mu_{\min}$ and $\mu_{\max}$ so that the anchor points generated in Algorithm \ref{alg:mpfa} land on the critical region.  This adaptive initialization enhances the efficiency and robustness of the Manifold Pareto front approximation method (Algorithm \ref{alg:mpfa}), making it applicable to a wide variety of problems without requiring manual tuning of the weight range.
\end{remark}
\begin{algorithm}
\caption{Riemannian Newton Method}
\label{alg:newton}
\begin{algorithmic}[1]
	\STATE \textbf{Input:} Manifold $\mathcal M$, retraction $R$, initial point $x_0\in \mathcal{M}$, objective function $\varphi_\mu$,
	$K\in \mathbb{N}$, tolerance $\varepsilon>0$.
	
	\STATE Set $x \coloneqq x_0$.
	
	\FOR{$k=1,\ldots,K$}
	\STATE Compute
	\[
	g \coloneqq \operatorname{grad}\varphi_\mu(x).
	\]
	\IF{$\|g\|_x\leq \varepsilon$}
	\STATE Stop the iteration.
	\ENDIF
	
	\STATE Solve the tangent-space Newton equation
	\[
	\operatorname{Hess}\varphi_\mu(x)[\xi]
	=
	-g,
	\qquad
	\xi\in T_x\mathcal M.
	\]

	\STATE Choose a step size $\alpha>0$ by Armijo line search starting at $\alpha = 1$.
	
	\STATE Update
	\[
	x \coloneqq R_x(\alpha \xi),
	\]
	where $R$ is a retraction on $\mathcal M$.
	\ENDFOR
	
	\STATE \textbf{return} $x$.
\end{algorithmic}
\end{algorithm}

\begin{algorithm}
\caption{\textsc{HermiteInterpolationOnManifold} \cite[Theorem 1]{z20})}
\label{alg:hermite-interp}
\begin{algorithmic}[1]
	\STATE Given: Manifold $\mathcal{M}$,
	\text{MuList} $\{\mu_i\} \subset \mathbb{R}$,
	\text{AnchorPoints} $\{x_i\} \subset \mathcal{M}$, 
	velocity vectors $v_i\in V$.
	
	\FOR{each consecutive pair $(\mu_i, \mu_{i+1})$}
	\STATE Choose $p  \coloneqq x_i$, $q  \coloneqq x_{i+1}$
	\STATE Compute $\Gamma_p \coloneqq \operatorname{Log}_q(p)$
	\STATE Compute $\hat{v}_p \coloneqq d(\operatorname{Log}_q)_p(v_p) \in T_q\mathcal{M}$ \cite[eq. (14)]{z20}
	\STATE Define cubic Hermite basis functions $a_0(\mu), b_0(\mu), b_1(\mu)$ \cite[Theorem 1]{z20} on $[\mu_i, \mu_{i+1}]$
	\STATE Build tangent space curve:
	\begin{equation}\label{hcurve}
		\tilde\Gamma_i(\mu) = a_0(\mu) \Gamma_p + b_0(\mu) \hat{v}_p + b_1(\mu) v_{i+1} \in T_q\mathcal{M}
	\end{equation}
	\STATE Map to manifold:
	\[
	\tilde\gamma_i(\mu) = \operatorname{Exp}_q(\tilde\Gamma_i(\mu))
	\]
	\ENDFOR
	\STATE Concatenate $ \tilde\gamma_i$ to form global $C^1$ curve $ \tilde\gamma$ \cite[Remark 2]{z20}
	\RETURN $ \tilde\gamma$
\end{algorithmic}
\end{algorithm}

\begin{remark}
The practical computation of $\hat{v}_p = d(\operatorname{Log}_q)_p(v_p)$ in Algorithm~\ref{alg:hermite-interp} can be performed numerically via finite differences as
\[
\hat{v}_p = \frac{(\operatorname{Log}_q \circ \operatorname{Exp}_p)(h v_p) - (\operatorname{Log}_q \circ \operatorname{Exp}_p)(-h v_p)}{2h} + \mathcal{O}(h^2),
\]
where $h>0$ is a small step size.  Assume that $p\in\mathcal I_q$. Since $\mathcal I_q$ is open,
$\operatorname{Exp}_p$ is continuous, and $\operatorname{Exp}_p(0)=p$, there exists a
neighborhood $V$ of $0\in T_p\mathcal M$ such that
$\operatorname{Exp}_p(V)\subset\mathcal I_q$.
Consequently, $\operatorname{Log}_q\circ\operatorname{Exp}_p:V\to T_q\mathcal M$
is smooth and
its derivative at $0$ is $d(\operatorname{Log}_q)_p$. 	
The resulting interpolant is globally $C^1$ by construction, as the derivative at the junction $\mu_{i+1}$ from the left coincides with the prescribed velocity $v_{i+1}$ (by the choice $\hat{v}_{i+1} = v_{i+1}$ in the subsequent interval). For further details, see \cite[Remark 2]{z20}.
\end{remark}

Now, we show that Algorithm~\ref{alg:mpfa} constructs an approximation of the Pareto front for the bi-objective optimization problem on Riemannian manifolds~\eqref{bo} under some assumptions.
First, the existence of solutions for the scalarization problem is presented. 

\begin{lemma}\label{rel_sol}
Let $\mu>0$ and $\bar{x} \in \mathcal{M}$ be a global minimizer of \eqref{sp}. Then, $\bar{x}$ is a properly efficient solution of \eqref{bo}.
\end{lemma}
\begin{proof}
The result follows from Lemma \ref{sp_pes}, taking $\mu_1=1$, $\mu_2=\mu$.
\end{proof}

Define 
\begin{align}
S_{\rm val}(\mu)&:=\min_{x\in \mathcal{M} } \varphi_\mu (x), \label{val}\\
S_{\rm opt}(\mu)&:=\{x\in \mathcal{M}| \varphi_\mu (x)=S_{\rm val}(\mu) \}. \label{val}
\end{align}

The following property is the counterpart of \cite[Theorem 3B.5, Example 3B.6]{dr09} for optimization over a compact Riemannian manifold, and it is useful in the theoretical analysis of Algorithm \ref{alg:mpfa}.

\begin{lemma}
Let  $\mathcal{M}$ be a compact Riemannian manifold and $f_1, f_2: \mathcal{M} \rightarrow \mathbb{R}$ be continuous functions. Then:

\item{\bf (i)} $\mu\mapsto 	S_{\rm val}(\mu)$ is continuous on $(0,\infty)$. 

\item{\bf (ii)} $\mu\mapsto 	S_{\rm opt}(\mu)$ is outer semicontinuous on $(0,\infty)$.
\end{lemma}

\begin{proof}
(i) Since $\mathcal M$ is compact and $f_2$ is continuous, there exists
$\hat L>0$ such that
\[
|f_2(x)|\leq \hat L,
\qquad \forall x\in\mathcal M.
\]
For any $\mu,\nu>0$, let
$x_\mu\in S_{\rm opt}(\mu)$ and
$x_\nu\in S_{\rm opt}(\nu)$. Then,
\[
\begin{aligned}
	S_{\rm val}(\mu)-S_{\rm val}(\nu)
	&\leq \varphi_\mu(x_\nu)-\varphi_\nu(x_\nu)\\
	&=(\mu-\nu)f_2(x_\nu)
	\leq \hat L|\mu-\nu|.
\end{aligned}
\]
Similarly, we obtain
\[
S_{\rm val}(\nu)-S_{\rm val}(\mu)
\leq \hat L|\mu-\nu|.
\]
It follows that
\[
|S_{\rm val}(\mu)-S_{\rm val}(\nu)|
\leq \hat L|\mu-\nu|,
\]
and therefore $S_{\rm val}$ is continuous.

(ii) We next prove the outer semicontinuity of $S_{\rm opt}$.
Let $\mu_k\to\bar\mu$ and let
$x_k\in S_{\rm opt}(\mu_k)$ with $x_k\to\bar x$.
Since
\[
\varphi_{\mu_k}(x_k)=S_{\rm val}(\mu_k),
\]
the continuity of $\varphi_\mu(x)$ with respect to $(\mu,x)$ and the
continuity of $S_{\rm val}$ yield
\[
\varphi_{\bar\mu}(\bar x)
=
S_{\rm val}(\bar\mu).
\]
Thus,
\[
\bar x\in S_{\rm opt}(\bar\mu).
\]
Therefore, $S_{\rm opt}$ is outer semicontinuous.
\end{proof}

\begin{corollary}\label{gam_con}
Let  $\mathcal{M}$ be a compact Riemannian manifold and $f_1, f_2: \mathcal{M} \rightarrow \mathbb{R}$ be continuous functions. If $\gamma(\mu)$ defined in~\eqref{fgamma} is a singleton
for  each $\mu>0$, then the mapping $\mu \mapsto \gamma(\mu)$ is continuous.
\end{corollary}	

Next, we present a key result that gives conditions under which  Step \ref{11sol} of Algorithm \ref{alg:mpfa} is well-defined.

\begin{lemma}\label{gamma_dif}
Assume that  the following conditions hold:   

\item{\bf (i)}  $\mathcal{M}$ is a smooth and compact Riemannian manifold;

\item{\bf (ii)} $f_1, f_2: \mathcal{M} \rightarrow \mathbb{R}$ are twice continuously differentiable;

\item{\bf (iii)} for every $\mu>0$, problem \eqref{sp} has a unique solution, i.e., $\gamma(\mu)$ is a singleton;

\item{\bf (iv)} for every $\mu>0$, the Riemannian Hessian $H_\mu(\gamma(\mu)):=\operatorname{Hess}\varphi_\mu(\gamma(\mu))$ is invertible on $T_{\gamma(\mu)}\mathcal{M}$.	

Then the mapping $\mu \mapsto \gamma(\mu)$ is continuously differentiable and its derivative at $\mu$ (a tangent vector in $T_{\gamma(\mu)}\mathcal{M}$, denoted by $\dot\gamma(\mu)$) is given by $\dot\gamma(\mu)=-{(H_\mu(\gamma(\mu)))}^{-1}\operatorname{grad} f_2(\gamma(\mu))$.
\end{lemma}

\begin{proof} Fix a $\bar\mu>0$ and set 
\[ \bar x:=\gamma(\bar\mu). \] 
Since $\bar x$ is a global minimizer of $\varphi_{\bar\mu}$, the first-order optimality condition gives 
\[ 
\operatorname{grad}\varphi_{\bar\mu}(\bar x)=0. 
\] 
Let $(U,\psi)$ be a smooth local coordinate chart around $\bar x$, where 
\[ \psi:U\to V\subset\mathbb R^d, \qquad d=\dim\mathcal M, \] 
and let $ \bar z:=\psi(\bar x)$. 
For $z\in V$, write 
$x=\psi^{-1}(z)$.  
We define the coordinate representation of the Riemannian gradient by 
\[ G:(0,\infty)\times V\to\mathbb R^d, \] 
\[ G(\mu,z) := d\psi_x\bigl(\operatorname{grad}\varphi_\mu(x)\bigr), \qquad x=\psi^{-1}(z). \] 
Since 
$ \varphi_\mu=f_1+\mu f_2 $
and $f_1,f_2\in C^2(\mathcal M)$, the vector field 
\[ (\mu,x)\longmapsto \operatorname{grad}\varphi_\mu(x) = \operatorname{grad} f_1(x)+\mu\operatorname{grad} f_2(x) \] 
is continuously differentiable. Since $\psi$ is smooth, it follows that $G$ is continuously differentiable. At $(\bar\mu,\bar z)$, we have 
\[ G(\bar\mu,\bar z) = d\psi_{\bar x} \bigl(\operatorname{grad}\varphi_{\bar\mu}(\bar x)\bigr) =0. \] 
We next show that the partial derivative 
\[ D_zG(\bar\mu,\bar z): \mathbb R^d\to\mathbb R^d \] is nonsingular. Let $\zeta\in\mathbb R^d$ and define \[ \xi := d\psi_{\bar x}^{-1}[\zeta] \in T_{\bar x}\mathcal M. \] 
Differentiating the local representation 
$z\longmapsto G(\mu,z)$ 
at $z=\bar z$, since 
$ \operatorname{grad}\varphi_{\bar\mu}(\bar x)=0$, 
we get
\[ D_zG(\bar\mu,\bar z)[\zeta] = d\psi_{\bar x} \left( \operatorname{Hess}\varphi_{\bar\mu}(\bar x)[\xi] \right). \] Equivalently, 
\[ D_zG(\bar\mu,\bar z) = d\psi_{\bar x} \circ \operatorname{Hess}\varphi_{\bar\mu}(\bar x) \circ d\psi_{\bar x}^{-1}. \] 
So $D_zG(\bar\mu,\bar z)$ is the coordinate representation of $\operatorname{Hess}\varphi_{\bar\mu}(\bar x)$. By assumption (iv), 
$\operatorname{Hess}\varphi_{\bar\mu}(\bar x)$
is invertible. Since $d\psi_{\bar x}$ is a linear isomorphism, it follows that 
$D_zG(\bar\mu,\bar z)$ 
is invertible. The classical implicit function theorem  \cite[Theorem~1B.1]{dr09} can be applied to 
$G(\mu,z)=0$ 
at $(\bar\mu,\bar z)$;
see, \cite[Lemmas~1.3.1-1.3.2]{j05}. 
Hence there exist an open interval $J_{\bar\mu}$ containing $\bar\mu$, an open neighborhood $W\subset V$ of $\bar z$, and a unique continuously differentiable mapping 
$z:J_{\bar\mu}\to W $ 
such that 
$ z(\bar\mu)=\bar z $ 
and 
$ G(\mu,z(\mu))=0$, for every $\mu\in J_{\bar\mu}$.  
We now identify this local implicit branch with the global solution mapping $\gamma$. By Corollary~\ref{gam_con}, the mapping $\mu\mapsto\gamma(\mu)$ is continuous. Hence, after possibly shrinking $J_{\bar\mu}$, we may assume that \[ \gamma(\mu)\in U \qquad\text{and}\qquad \psi(\gamma(\mu))\in W \] 
for every $\mu\in J_{\bar\mu}$. For every $\mu\in J_{\bar\mu}$, $\gamma(\mu)$ is the global minimizer of $\varphi_\mu$. Therefore, 
$\operatorname{grad}\varphi_\mu(\gamma(\mu))=0$ and  
$G\bigl(\mu,\psi(\gamma(\mu))\bigr)=0$.
Both $z(\mu)$ and $\psi(\gamma(\mu))$ belong to $W$ and solve 
$G(\mu,\cdot)=0$.
By the local uniqueness in the implicit function theorem, 
\[ \psi(\gamma(\mu))=z(\mu), \qquad \mu\in J_{\bar\mu}. \] 
Thus, 
\[ \gamma(\mu) = \psi^{-1}(z(\mu)), \qquad \mu\in J_{\bar\mu}. \] 
Since $z$ and $\psi^{-1}$ are continuously differentiable, $\gamma$ is continuously differentiable on $J_{\bar\mu}$. Since $\bar\mu>0$ was arbitrary, we conclude that 
\[ \gamma\in C^1((0,\infty),\mathcal M). \] 
It remains to derive the expression for $\dot\gamma(\mu)$. For every $\mu>0$, the first-order optimality condition gives 
\[ \operatorname{grad}\varphi_\mu(\gamma(\mu))=0. \] 
Consider the parameter-dependent vector field 
\[ X(\mu,x) := \operatorname{grad}\varphi_\mu(x) = \operatorname{grad} f_1(x)+\mu\operatorname{grad} f_2(x). \] 
Taking the covariant derivative of 
\[ X(\mu,\gamma(\mu))=0 \] 
with respect to $\mu$ yields 
\[ 0 = \frac{D}{d\mu} X(\mu,\gamma(\mu)). \] 
By the chain rule for a parameter-dependent vector field, 
\[ \frac{D}{d\mu} X(\mu,\gamma(\mu)) =  \frac{\partial X}{\partial\mu} (\mu,\gamma(\mu))+ \nabla_{\dot\gamma(\mu)} X(\mu,\cdot). \] The second term satisfies 
\[ \nabla_{\dot\gamma(\mu)} \operatorname{grad}\varphi_\mu = \operatorname{Hess}\varphi_\mu(\gamma(\mu)) [\dot\gamma(\mu)], \] while 
\[ \frac{\partial X}{\partial\mu} (\mu,\gamma(\mu)) = \operatorname{grad} f_2(\gamma(\mu)). \] 
Therefore, \[ \operatorname{Hess}\varphi_\mu(\gamma(\mu)) [\dot\gamma(\mu)] + \operatorname{grad} f_2(\gamma(\mu)) = 0. \] Finally, assumption (iv) implies that $\operatorname{Hess}\varphi_\mu(\gamma(\mu))$ is invertible, and hence \[  \dot\gamma(\mu) = - \bigl(\operatorname{Hess}\varphi_\mu(\gamma(\mu))\bigr)^{-1} \operatorname{grad} f_2(\gamma(\mu)).  \]
This completes the proof. 
\end{proof}

\section{Error and Convergence Analysis}\label{sec4}

In  Algorithm \ref{alg:mpfa}, a finite set of parameter values $\{ \mu_i \}_{i=1}^N$ is chosen, and corresponding Pareto critical points $\gamma(\mu_i)$ are computed. A $C^1$ curve $\tilde{\gamma}$ on $\mathcal{M}$ is then constructed by Hermite interpolation between these critical points (anchor points), using also velocity vectors $v_i = \dot{\gamma}(\mu_i)$. 
A natural and important question is: How close is the interpolating curve $\tilde{\gamma}$ to the true Pareto curve $\gamma$? 

In this section, we provide an upper bound for the uniform approximate error in terms of the maximum spacing between the anchor points.
First, on each interval $I_i=[\mu_i,\mu_{i+1}]$, we pull back the true Pareto curve $\gamma$
to the tangent space $T_q\mathcal M$ by means of the logarithmic map at the base point
$q=\gamma(\mu_{i+1})$, and define $\Gamma(\mu)=\mathrm{Log}_q(\gamma(\mu))$.
This transforms the approximation of $\gamma$ into a classical interpolation problem in a linear space.
Using the uniform continuity of $\dot \Gamma_i$, we derive a local error estimate for the cubic Hermite interpolation in the tangent space in terms of the modulus of continuity of $\dot \Gamma_i$.
Second, we transfer this estimate to the manifold by applying Theorem~1 of \cite{j25}.

We transfer the tangent-space error back to the manifold by the exponential map.
Combining both steps yields a global uniform bound of the form
$\Delta_{\max}$ and $\eta$, such that
$$
\sup_{\mu}
\mathrm{dist}\bigl(\gamma(\mu),\tilde{\gamma}(\mu)\bigr)=
\mathcal{O}\!\left(
\Delta_{\max}\omega(\Delta_{\max})
+\eta\Delta_{\max}
\right),$$
which further implies convergence of the approximate Pareto front in the objective space.

Partition $\mu_0<\mu_1<\cdots<\mu_N$ with $\Delta_i \coloneqq \mu_{i+1}-\mu_i$ and
\[
\Delta_{\max} \coloneqq \max_{1 \leq i \leq N-1} \Delta_i.
\]

\begin{lemma} \label{lem:Gamma-regularity}
Assume that the conditions of Lemma \ref{gamma_dif} hold.   
For each interpolation interval $I_i=[\mu_i,\mu_{i+1}]$, choose the basepoint 
\[
q := \gamma(\mu_{i+1}),
\]
and define the mapping $\Gamma(\mu) := \mathrm{Log}_q(\gamma(\mu))$ for $\mu\in I_i$. 
Then, for sufficiently small $\Delta_{\max}$, the following statements hold:

\item{\bf (i)}
The mapping $\Gamma$ belongs to $C^1(I_i,T_q\mathcal M)$, and its derivative satisfies
\[
\dot\Gamma(\mu) = D(\mathrm{Log}_q)(\gamma(\mu))\,[\dot\gamma(\mu)].
\]

\item{\bf (ii)}
$ \dot\Gamma$ is uniformly continuous on the interval $I_i$, 
\begin{equation*}\label{omega_Gam}
	\omega_{ \dot\Gamma}(t)
	:=\sup_{\substack{\mu,\nu\in[\mu_{\min},\mu_{\max}]\\ |\mu-\nu|\le t}}
	\| \dot\Gamma(\mu)- \dot\Gamma(\nu)\|_q
\end{equation*}
is finite for all $t\ge 0$, and satisfies
$
\lim_{t\downarrow 0} \omega_{ \dot\Gamma}(t) = 0$.

\end{lemma}

\begin{proof}
We prove each item in turn.

\item{\bf (i)} In the interior set $\mathcal{I}_q$, the map $\mathrm{Log}_q$ is $C^\infty$.  
Since
$\gamma\in C^1(I_i,\mathcal M)$ by 
Lemma~\ref{gamma_dif} and there exists $\Delta_{\max}>0$ such that
$
\gamma(I_i)\subset \mathcal{I}_q
$. Then $\Gamma=\mathrm{Log}_q\circ\gamma$ belongs to $C^1(I_i,T_q\mathcal M)$ by the
chain rule. Moreover, for every $\mu\in I_i$, the derivative is given by the
chain rule as
\[
\dot\Gamma(\mu)
= D(\mathrm{Log}_q)\big(\gamma(\mu)\big)\big[\dot\gamma(\mu)\big],
\]
where $D(\mathrm{Log}_q)(\gamma(\mu)) : T_{\gamma(\mu)}\mathcal M \to T_q\mathcal M$
denotes the differential of $\mathrm{Log}_q$ at the point $\gamma(\mu)$.

\item{\bf (ii)} The map $ \dot\Gamma$ is continuous by (a),  and the interval $I_i$ is compact. Hence $\Gamma'$ is uniformly
continuous on $I_i$. 
Therefore, the function
\[
\omega_{\dot\Gamma}(t)
:=\sup_{\substack{\mu,\nu\in I_i\\|\mu-\nu|\le t}}
\|\dot\Gamma(\mu)-\dot\Gamma(\nu)\|_q
\]
is finite for every $t\ge0$, and satisfies
$\lim_{t\downarrow0}\omega_{\dot\Gamma}(t)=0$, which is exactly the assertion in (b).
\end{proof}

\begin{theorem}\label{herror}
Assume that the conditions of Lemma~\ref{lem:Gamma-regularity} hold.
For each $i$, let $I_i=[\mu_i,\mu_{i+1}]$, $
q_i:=\gamma(\mu_{i+1})$,
$\Gamma_i(\mu):=\mathrm{Log}_{q_i}(\gamma(\mu))$,
$\mu\in I_i$.
Let $\tilde{\Gamma}_i:I_i\to T_{q_i}\mathcal M$ be the cubic
Hermite interpolant used in Algorithm~\ref{alg:hermite-interp}, and define $\tilde{\gamma}_i(\mu)
:=
\mathrm{Exp}_{q_i}\bigl(\tilde{\Gamma}_i(\mu)\bigr)$.
Suppose that the endpoint tangent data used in constructing
$\tilde{\Gamma}_i$ satisfy
$
\|\widehat v_i^L-\dot\Gamma_i(\mu_i)\|_{q_i}\leq\eta$,
$\|\widehat v_i^R-\dot\Gamma_i(\mu_{i+1})\|_{q_i}\leq\eta$,
for some $\eta\geq0$. Define
$
\omega(t)
:=
\max_{0\leq i\leq N-1}
\sup_{\substack{\mu,\nu\in I_i\\|\mu-\nu|\leq t}}
\|\dot\Gamma_i(\mu)-\dot\Gamma_i(\nu)\|_{q_i}
$. Then, for sufficiently small $\Delta_{\max}$ and $\eta$,
there exists a constant $C>0$, independent of the partition,
$\Delta_{\max}$ and $\eta$, such that
\[
\sup_{\mu\in[\mu_{\min},\mu_{\max}]}
\mathrm{dist}\bigl(\gamma(\mu),\tilde{\gamma}(\mu)\bigr)
\leq
C\left[
\Delta_{\max}\omega(\Delta_{\max})
+
\eta\Delta_{\max}
\right].	
\]
In particular, if the exact endpoint velocities are used, i.e.,
$\eta=0$, then
\[
\sup_{\mu\in[\mu_{\min},\mu_{\max}]}
\mathrm{dist}\bigl(\gamma(\mu),\tilde{\gamma}(\mu)\bigr)
\leq
C\Delta_{\max}\omega(\Delta_{\max}).
\]
\end{theorem}

\begin{proof}
By Lemma~\ref{lem:Gamma-regularity}, for all sufficiently small
$\Delta_{\max}$, the mapping
\[
\Gamma_i(\mu)=\mathrm{Log}_{q_i}(\gamma(\mu)),
\qquad \mu\in I_i,
\]
is well defined and belongs to
$C^1(I_i,T_{q_i}\mathcal M)$ for every $i$.
Moreover, the family $\{\dot \Gamma_i\}$ is uniformly continuous in the
sense that
$
\lim_{t\downarrow 0} \omega_{ \dot\Gamma}(t) = 0$.

First, we apply a classical analysis~\cite[Section~2]{bs68} of cubic Hermite interpolants to estimate the interpolation error in the tangent space.
Fix $i$ and write
\[
a=\mu_i,\qquad
b=\mu_{i+1},\qquad
h=b-a=\Delta_i.
\]	
Let $H_i\Gamma_i$ denote the cubic Hermite interpolant constructed
from the exact endpoint values
$\Gamma_i(a),\, \Gamma_i(b)$,
and the exact endpoint derivatives
$\dot \Gamma_i(a),\, \dot \Gamma_i(b)$.
Define the affine function
\[
\ell_i(\mu)
:=
\Gamma_i(a)+(\mu-a)\dot \Gamma_i(a)
\]
and set
\[
r_i(\mu):=\Gamma_i(\mu)-\ell_i(\mu).
\]
Since cubic Hermite interpolation reproduces affine functions,
\[
\Gamma_i-H_i\Gamma_i
=
r_i-H_ir_i.
\]

For every $\mu\in I_i$,
\[
\begin{aligned}
r_i(\mu)
&=
\Gamma_i(\mu)-\Gamma_i(a)
-(\mu-a)\dot \Gamma_i(a)\\
&=
\int_a^\mu
\left(\dot \Gamma_i(s)-\dot\Gamma_i(a)\right)\,ds.
\end{aligned}
\]
Hence,
\[
\|r_i(\mu)\|_{q_i}
\leq
h\,\omega(h).
\]
Furthermore,
$r_i(a)=0$,
$\dot r_i(a)=0$,
$\|r_i(b)\|_{q_i}\leq h\omega(h)$, and $\|\dot r_i(b)\|_{q_i}\leq\omega(h)$.

Let
\[
\theta=\frac{\mu-a}{h}\in[0,1].
\]
We obtain the cubic Hermite interpolant
\begin{align*}
H_ir_i(\mu)
&=h_{00}(\theta)r_i(a)+
h\,h_{10}(\theta)\dot r_i(a)+
h_{01}(\theta)r_i(b)+
h\,h_{11}(\theta)\dot r_i(b)\\
&=h_{01}(\theta)r_i(b)+
h\,h_{11}(\theta)\dot r_i(b),
\end{align*}
where
$h_{00}(\theta)=2\theta^3-3\theta^2+1$,
$h_{10}(\theta)=\theta^3-2\theta^2+\theta$,
$h_{01}(\theta)=-2\theta^3+3\theta^2$,
$h_{11}(\theta)=\theta^3-\theta^2$.
Since $
|h_{01}(\theta)|\leq1, |h_{11}(\theta)|\leq1, \theta\in[0,1]$,
we obtain
\[
\|H_ir_i(\mu)\|_{q_i}
\leq2h\omega(h).
\]
Consequently,
\[
\begin{aligned}
\|\Gamma_i(\mu)-H_i\Gamma_i(\mu)\|_{q_i}
=
\|r_i(\mu)-H_ir_i(\mu)\|_{q_i}
\leq
\|r_i(\mu)\|_{q_i}+\|H_ir_i(\mu)\|_{q_i}
\leq
3h\omega(h).
\end{aligned}
\]
Thus,
\begin{equation}
\label{eq:exactHermite}
\|\Gamma_i(\mu)-H_i\Gamma_i(\mu)\|_{q_i}
\leq
3\Delta_i\omega(\Delta_i).
\end{equation}

We next account for the possible error in the endpoint tangent data.
The interpolants $H_i\Gamma_i$ and $\tilde{\Gamma}_i$ have the
same endpoint values and differ only in their endpoint derivatives.
Hence, by the cubic Hermite representation,
\[
\begin{aligned}
H_i\Gamma_i(\mu)-\tilde{\Gamma}_i(\mu)
&=
h h_{10}(\theta)
\bigl(\dot \Gamma_i(a)-\widehat v_i^L\bigr)+
h h_{11}(\theta)
\bigl(\dot \Gamma_i(b)-\widehat v_i^R\bigr).
\end{aligned}
\]
Using
$
|h_{10}(\theta)|\leq1$,
$
|h_{11}(\theta)|\leq1$,
and the assumed endpoint derivative error bounds, we obtain
\[
\|H_i\Gamma_i(\mu)-\tilde{\Gamma}_i(\mu)\|_{q_i}
\leq2\eta\Delta_i.
\]
Combining this estimate with \eqref{eq:exactHermite} yields
\begin{equation}
\label{eq:tang-error}
\|\Gamma_i(\mu)-\tilde{\Gamma}_i(\mu)\|_{q_i}
\leq
3\Delta_i\omega(\Delta_i)+2\eta\Delta_i.
\end{equation}

We next verify that both $\Gamma_i$ and
$\tilde{\Gamma}_i$ remain in a common tangent normal-coordinate
domain. Since $\gamma\in C^1([\mu_{\min},\mu_{\max}],\mathcal M)$,
there exists $L_\gamma>0$ such that
\[
\|\dot\gamma(\mu)\|
\leq L_\gamma,
\qquad
\mu\in[\mu_{\min},\mu_{\max}].
\]
For $\mu\in I_i$,
\[
\begin{aligned}
\mathrm{dist}(q_i,\gamma(\mu))
&\leq
\int_\mu^{\mu_{i+1}}
\|\dot\gamma(s)\|\,ds\\
&\leq
L_\gamma\Delta_i.
\end{aligned}
\]
Since $\gamma(I_i)$ lies in the interior set of $q_i$ for sufficiently
small $\Delta_{\max}$, we have
\[
\|\Gamma_i(\mu)\|_{q_i}
=
\mathrm{dist}(q_i,\gamma(\mu))
\leq L_\gamma\Delta_i.
\]
It follows from \eqref{eq:tang-error} that
\[
\begin{aligned}
\|\tilde{\Gamma}_i(\mu)\|_{q_i}
&\leq
\|\Gamma_i(\mu)\|_{q_i}
+
\|\Gamma_i(\mu)-\tilde{\Gamma}_i(\mu)\|_{q_i}\\
&\leq
L_\gamma\Delta_i
+
3\Delta_i\omega(\Delta_i)
+
2\eta\Delta_i.
\end{aligned}
\]
Since $\omega(t)\to0$ as $t\downarrow0$,
the right-hand side tends uniformly to zero as
$\Delta_{\max}\to0$ and $\eta\to0$.

Because $\mathcal M$ is compact, one may choose a radius $\rho>0$,
independent of $q\in\mathcal M$, such that for every $q\in\mathcal M$
the ball $B_{T_q\mathcal M}(0,\rho)$
is contained in a tangent normal-coordinate domain of $q$.
Therefore, for sufficiently small $\Delta_{\max}$ and $\eta$,
\[
\Gamma_i(I_i)\cup\tilde{\Gamma}_i(I_i)
\subset
B_{T_{q_i}\mathcal M}(0,\rho)
\]
for every $i$. In particular,
\[
\mathrm{Log}_{q_i}(\tilde{\gamma}_i(\mu))
=
\tilde{\Gamma}_i(\mu).
\]

We now transfer the tangent-space error to the manifold.
Since $\mathcal M$ is compact, its sectional curvature is bounded
from below, thus there exists $H\in\mathbb R$ such that $K(\sigma)\geq H$
for every tangent two-plane $\sigma$ of $\mathcal M$. Set
\[
\varepsilon_i
:=
\sup_{\mu\in I_i}
\|\Gamma_i(\mu)-\tilde{\Gamma}_i(\mu)\|_{q_i}.
\]
By \eqref{eq:tang-error},
\[
\varepsilon_i
\leq
3\Delta_i\omega(\Delta_i)
+
2\eta\Delta_i.
\]
The error-transfer estimate of
\cite[Theorem~1]{j25}
can now be applied on the common normal-coordinate domain with $R=I_i$,
$f=\gamma|_{I_i}$,
$g=\Gamma_i$,
$\widehat g=\tilde{\Gamma}_i$,
$\widehat f=\tilde{\gamma}_i$.
It follows that there exists a constant $C_1>0$, 
such that
\[
\mathrm{dist}\bigl(\gamma(\mu),\tilde{\gamma}_i(\mu)\bigr)
\leq C_1\varepsilon_i,
\qquad
\mu\in I_i.
\]
Therefore,
\[
\begin{aligned}
\mathrm{dist}\bigl(\gamma(\mu),\tilde{\gamma}_i(\mu)\bigr)
&\leq
C_1
\left[
3\Delta_i\omega(\Delta_i)
+
2\eta\Delta_i
\right]\\
&\leq
C_1
\left[
3\Delta_{\max}\omega(\Delta_{\max})
+
2\eta\Delta_{\max}
\right].
\end{aligned}
\]     
We take the supremum over
all interpolation intervals:
\[
\sup_{\mu\in[\mu_{\min},\mu_{\max}]}
\mathrm{dist}\bigl(\gamma(\mu),\tilde{\gamma}(\mu)\bigr)
\leq
C
\left[
\Delta_{\max}\omega(\Delta_{\max})
+
\eta\Delta_{\max}
\right]
\]
for a constant $C>0$ independent of the partition.

If the exact endpoint velocities are used, then $\eta=0$, and hence
\[
\sup_{\mu\in[\mu_{\min},\mu_{\max}]}
\mathrm{dist}\bigl(\gamma(\mu),\tilde{\gamma}(\mu)\bigr)
\leq
C\Delta_{\max}\omega(\Delta_{\max}).
\]
\end{proof}

Moreover, we say that the proposed method (Algorithm \ref{alg:mpfa}) is \emph{convergent} in the objective space 
if the approximate Pareto front curve
$\bigl(f_1(\tilde{\gamma}(\mu)),f_2(\tilde{\gamma}(\mu))\bigr)$ converges to  $\bigl(f_1(\gamma(\mu)),f_2(\gamma(\mu))\bigr)$ when $\Delta_{\max}\to 0$, namely,
\[
\lim_{\Delta_{\max}\to 0}
\sup_{\mu\in[\mu_{\min},\mu_{\max}]}
\left\|
\bigl(f_1(\gamma(\mu)),f_2(\gamma(\mu))\bigr)
-
\bigl(f_1(\tilde{\gamma}(\mu)),f_2(\tilde{\gamma}(\mu))\bigr)
\right\|
=0.
\] 
The convergence of Algorithm \ref{alg:mpfa} is proved in the next theorem.

\begin{theorem} \label{thm:mpfa-convergence}
Assume that conditions of Theorem \ref{herror} hold.
Let the objective functions
$f_1,f_2$ be Lipschitz continuous with Lipschitz constants $L_{f_1},L_{f_2}>0$.
Then 
\begin{equation}\label{eq:mpfa-convergence}
	\begin{split}
		&\sup_{\mu\in[\mu_{\min},\mu_{\max}]}\big\|
		(f_1(\gamma(\mu)),f_2(\gamma(\mu))) - 
		(f_1(\tilde\gamma(\mu)),f_2(\tilde\gamma(\mu)))\big\| \\
		\le& \sqrt{L_{f_1}^2 + L_{f_2}^2} C \left( \Delta_{\max}\omega(\Delta_{\max}) + \eta\Delta_{\max} \right),
	\end{split}
\end{equation}
where  $\omega(\cdot)$ is defined in Lemma 
\ref{lem:Gamma-regularity}.

In particular, as $\Delta_{\max}\to 0$, $(f_1(\tilde\gamma(\mu)),f_2(\tilde\gamma(\mu)))$ generated by  Algorithm \ref{alg:mpfa} converges to the true Pareto front.
\end{theorem}

\begin{proof}
By Theorem \ref{herror}, we have
\begin{equation}\label{eq:hcon}
	\sup_{\mu\in[\mu_{\min},\mu_{\max}]}\mathrm{dist}\big(\gamma(\mu),\tilde\gamma(\mu)\big)
	\le C \left( \Delta_{\max}\omega(\Delta_{\max}) + \eta\Delta_{\max} \right).
\end{equation}

Since $f_1,f_2$ are Lipschitz with constants $L_{f_1},L_{f_2}$,
\[
\begin{aligned}
	&\|(f_1(\gamma(\mu)),f_2(\gamma(\mu))) - (f_1(\tilde\gamma(\mu)),f_2(\tilde\gamma(\mu)))\| \\
	\le& \sqrt{L_{f_1}^2 + L_{f_2}^2} \,\mathrm{dist}(\gamma(\mu),\tilde\gamma(\mu))\\
	\le& \sqrt{L_{f_1}^2 + L_{f_2}^2}C \left( \Delta_{\max}\omega(\Delta_{\max}) + \eta\Delta_{\max} \right).
\end{aligned}
\]  

As $\Delta_{\max}\to 0$, the right-hand side of the above inequality tends to zero, 
implying uniform convergence of $\left(f_1,f_2\right)(\tilde\gamma(\mu))$ to $(f_1,f_2)(\gamma(\mu))$. 
Hence, Algorithm \ref{alg:mpfa} converges to the true Pareto front.	
\end{proof}

\section{Numerical Experiments}\label{sec5}

In this section, we present numerical experiments to illustrate the effectiveness and competitiveness of the proposed MPFA method (Algorithm \ref{alg:mpfa}). 
Several test problems are considered to evaluate the ability of the method to approximate Pareto fronts on manifolds. The performance of the algorithm is assessed in terms of approximation accuracy and computational efficiency. The numerical results demonstrate that the proposed approach can provide accurate approximations of Pareto fronts without choosing multiple initial points. To further evaluate the accuracy of the proposed Pareto front approximation framework, we compare several interpolation strategies for reconstructing the front from sampled points. In particular, we consider  two other representative interpolation methods on manifolds: the Piecewise Geodesic interpolation  and the blended cubic spline with $\lambda=10$ \cite{gm19}. The three interpolation methods are constructed from the same set of anchor points generated by Algorithm~\ref{alg:mpfa}. Moreover, we demonstrate the applicability of the proposed framework to a practical problem arising in sparse principal component analysis (sparse PCA). In this setting, the sparse PCA problem is reformulated as a bi-objective optimization problem on the sphere manifold, and the MPFA method is used to approximate its Pareto front.

\subsection{Approximation Accuracy on Test Problems}\label{experiment1}

To evaluate the approximation performance of the proposed method, we consider six test problems defined on the unit sphere (see Appendix \ref{appendix}).  We present an algorithm framework (Algorithm \ref{alg:evaluate}) for an accuracy evaluation. 
\begin{algorithm}[H]
\caption{Accuracy Evaluation}
\label{alg:evaluate}
\begin{algorithmic}[1]
	\STATE \textbf{Given:}
	Manifold $\mathcal{M}$, objective functions $f_1,f_2:\mathcal{M}\to\mathbb{R}$,
	Riemannian gradients $\mathrm{grad}\,f_1$, $\mathrm{grad}\,f_2$,
	$\hat N\in\mathbb{N}$, 
	tolerance $\epsilon>0$,
	initial point $x_0 \in \mathcal{M}$, set of integers $\mathcal{N}$

	\STATE Determine $\mu_{\min},\mu_{\max}$ adaptively 
	using Algorithm~\ref{alg:adaptive-mu}

	\FOR{$j=1$ to $\hat N$}
	\STATE $\hat\mu_j \gets 
	\mu_{\min} + (\mu_{\max}-\mu_{\min})\frac{j-1}{\hat N-1}$
	\IF{$j=1$} 
	\STATE Compute $$x_1=
	\argmin_{x\in\mathcal M}\varphi_{\hat\mu_1}(x),$$
	using a Riemannian steepest descent method (see \texttt{Manopt}) 
	\ELSE
	\STATE Compute $$x_j=
	\argmin_{x\in\mathcal M}\varphi_{\hat\mu_j}(x),$$
	using a Riemannian Newton method (Algorithm~\ref{alg:newton}) initialized at $x_{j-1}$
	\ENDIF
	\STATE $r(\hat\mu_{j})=x_j$ 
	\ENDFOR

	\FOR{$N\in\mathcal N$}
	\STATE $\tilde{\gamma}_N$ by Algorithm \ref{alg:mpfa} 
	
	\STATE $E_N = \max_{j\in\{1,\cdots,\hat N\}}{\rm dist}( r(\hat\mu_j), \tilde{\gamma}_N(\hat \mu_j) )$
	\ENDFOR
	
\end{algorithmic}
\end{algorithm}
In addition, the proposed MPFA algorithm is first used to generate a sequence of points approximating the Pareto front on the manifold. Then the three interpolation methods are applied to reconstruct a continuous approximation of the Pareto front. The approximation error is evaluated by the maximum Riemannian
distance between the interpolated solution curve and the reference solution curve. For each test problem, we compute the approximation error using the MPFA method with $\hat N=1000$ and $N=50$.
The approximation error is
measured by
\begin{equation}\label{ERN}
E_N = \max_{j=1,\ldots,\hat N} \mathrm{dist}\!\left(r(\hat\mu_j),\tilde{\gamma}_N(\hat{\mu}_j)\right),
\end{equation}
where $\{r(\hat\mu_j)\}$ denotes reference solutions obtained by solving the scalarized problems and $\{\tilde{\gamma}_N(\hat{\mu}_j)\}$ is points on the approximate Pareto front.

Table \ref{tab:interp_error} reports the approximation errors for $N=50$ obtained by the three interpolation methods for six representative test problems. The results indicate that Hermite interpolation generally provides the highest accuracy. Therefore,  Algorithm \ref{alg:mpfa} is effective for constructing accurate Pareto front approximations on manifolds. In the next subsection, we conduct an experiment using the Hermite interpolation method.

\begin{table}[htbp]
\centering
\caption{Approximation errors of different interpolation methods with $N=50$.}
\label{tab:interp_error}
\begin{tabular}{c c c c}
	\hline
	\shortstack{Problems\\(see {\rm Appendix} \ref{appendix})} & Piecewise Geodesic & Blended cubic spline & Hermite \\
	\hline
	Problem 1 & $8.870\times10^{-4}$ & $4.855\times10^{-4}$ & $\mathbf{1.390\times10^{-4}}$ \\
	Problem 2 & $2.214\times10^{-3}$ & $6.807\times10^{-4}$ & $\mathbf{8.378\times10^{-6}}$ \\
	Problem 3 & $5.028\times10^{-4}$ & $2.712\times10^{-4}$ & $\mathbf{3.485\times10^{-6}}$ \\
	Problem 4 & $3.882\times10^{-3}$ & $2.133\times10^{-3}$ & $\mathbf{6.031\times10^{-5}}$ \\
	Problem 5 & $1.918\times10^{-3}$ & $1.226\times10^{-3}$ & $\mathbf{2.787\times10^{-6}}$ \\
	Problem 6 & $6.035\times10^{-4}$ & $3.262\times10^{-4}$ & $\mathbf{5.201\times10^{-6}}$ \\
	\hline
\end{tabular}
\end{table}

\subsection{Application to Sparse Principal Component Analysis}

Principal component analysis (PCA) is a fundamental data processing technique whose core objective is to identify a low-dimensional representation of a given dataset. Such representations are widely used in data denoising, visualization, and pattern recognition. However, classical PCA often produces dense loading vectors that are difficult to interpret in many applications.

To address this limitation, sparse PCA has been proposed. Sparse PCA aims to compute principal components with only a small number of nonzero entries while preserving as much variance as possible.
Given a data matrix $A\in \mathbb{R}^{m\times n}$, the sparse PCA problem can be formulated as
\begin{equation}\label{spca}
\min_{X\in \text{St}(n,r)} 
-\text{trace}(X^\top A^\top A X) + \delta \|X\|_1,
\end{equation}
where $\text{St}(n,r)=\{X\in\mathbb{R}^{n\times r}:X^\top X=I_r\}$ denotes the Stiefel manifold and $\delta$ is a regularization parameter controlling the sparsity level.

When $r=1$, problem \eqref{spca} reduces to
\begin{equation}\label{sspca}
\min_{x\in\mathbb{S}^{n-1}} 
-x^\top A^\top A x + \delta \|x\|_1 ,
\end{equation}
where $\mathbb{S}^{n-1}$ denotes the unit sphere.

It is worth noting that the performance of sparse PCA is highly sensitive to the choice of the regularization parameter $\delta$, which controls the trade-off between variance maximization and sparsity promotion. However, there is no exact or universal rule for selecting an optimal value of $\delta$, and an inappropriate choice may lead to either overly dense components with poor interpretability or excessively sparse solutions that fail to capture sufficient variance. Consequently, the effectiveness of sparse PCA methods can be significantly degraded by an inappropriate parameter setting.

To avoid this difficulty, the sparse PCA problem can be reformulated from a bi-objective optimization perspective, thereby eliminating the need for manually tuning the regularization parameter. From this viewpoint, sparse PCA naturally involves two conflicting objectives that should be optimized simultaneously: maximizing the variance explained by the principal components and enforcing sparsity of the loading vectors. Treating these objectives separately leads to a bi-objective sparse PCA formulation, in which the trade-off between variance preservation and sparsity is explicitly characterized by the Pareto front rather than implicitly controlled by a scalar regularization parameter. Specifically, the bi-objective sparse PCA problem can be formulated as
\begin{equation}\label{bospca1}
\min_{x\in \mathbb{S}^{n-1}}
\Bigl(	
-x^\top A^\top A x,\ \|x\|_1
\Bigr).
\end{equation}
According to the conditions of Lemma \ref{gamma_dif}, we introduce the following smooth approximation of the term $\|x\|_1$, 
\begin{equation}\label{eq:smooth_l1}
\|x\|_{1, \alpha}
\;:=\;
\sum_{i=1}^n
\frac{1}{\alpha} \, \ln\!\left( \exp(\alpha x_i) + \exp(-\alpha x_i) \right).
\end{equation}

An empirical study is conducted to investigate the behavior of Pareto front approximation methods for bi-objective sparse PCA, such as \cite{sa24}, where the data matrix is designed as 
\(
A = \Sigma U^\top \in \mathbb{R}^{3\times 6},
\)
where 
\(
\Sigma = \operatorname{diag}([20, 0.1, 0.05])
\)
and 
\(
U = \begin{bmatrix} \mathbb{I}_3 \\ \mathbf{0}_3 \end{bmatrix} \in \mathbb{R}^{6\times 3}.
\)
Here \(\mathbb{I}_3\) denotes the \(3\times 3\) identity matrix and \(\mathbf{0}_3\) the \(3\times 3\) zero matrix, this choice yields
\[
A^\top A = \begin{bmatrix} \Sigma^2 & \mathbf{0}_3 \\ \mathbf{0}_3 & \mathbf{0}_3 \end{bmatrix}.
\]
In this example, the sparse eigenvector exactly recovers the dominant direction while the remaining components are identically zero, providing a clear ground truth for validating algorithmic performance.

The data matrix is slightly perturbed to generate a nontrivial test instance. Specifically, we set 
\(
A = A + 0.1R_1
\)
where the entries of 
\(R_1\in\mathbb{R}^{3\times 6}\)  
is independently sampled from the standard normal distribution \(\mathcal{N}(0,1)\). This perturbation preserves the essential structure while removing the exact degeneracy, making the problem more representative of practical situations.
In the numerical comparisons, we employ the retraction
\[
R_x(\eta_x) = \frac{x+\eta_x}{\|x+\eta_x\|},\qquad 
x\in\mathbb{S}^{n-1},\;\eta_x\in\mathrm{T}_x\mathbb{S}^{n-1},
\]
on the sphere manifold. 
\begin{figure}[htbp]
\centering
\begin{subfigure}
	\centering
	\includegraphics[width=0.48\textwidth]{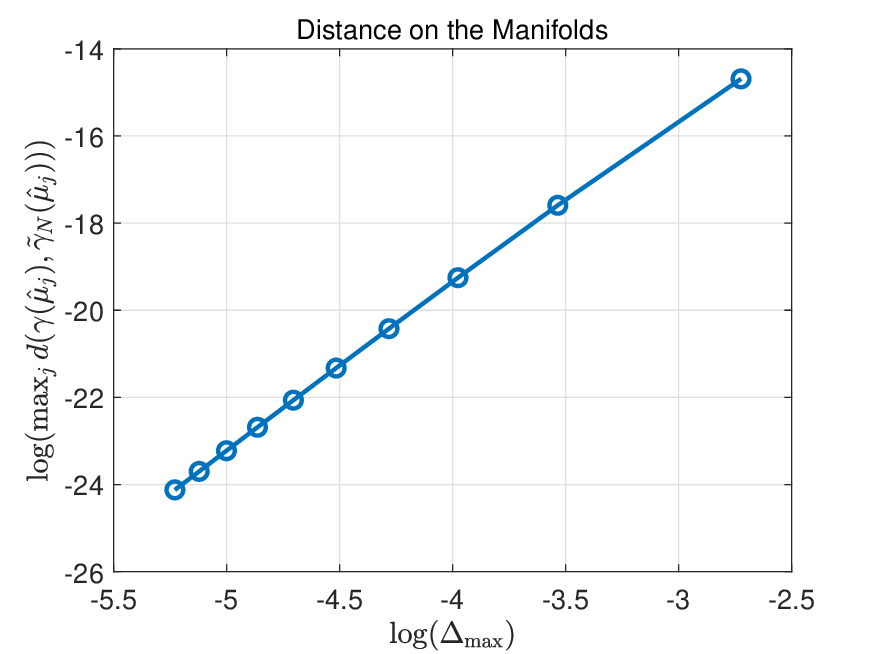}
	\label{fig:sub1}
\end{subfigure}
\hfill
\begin{subfigure}
	\centering
	\includegraphics[width=0.48\textwidth]{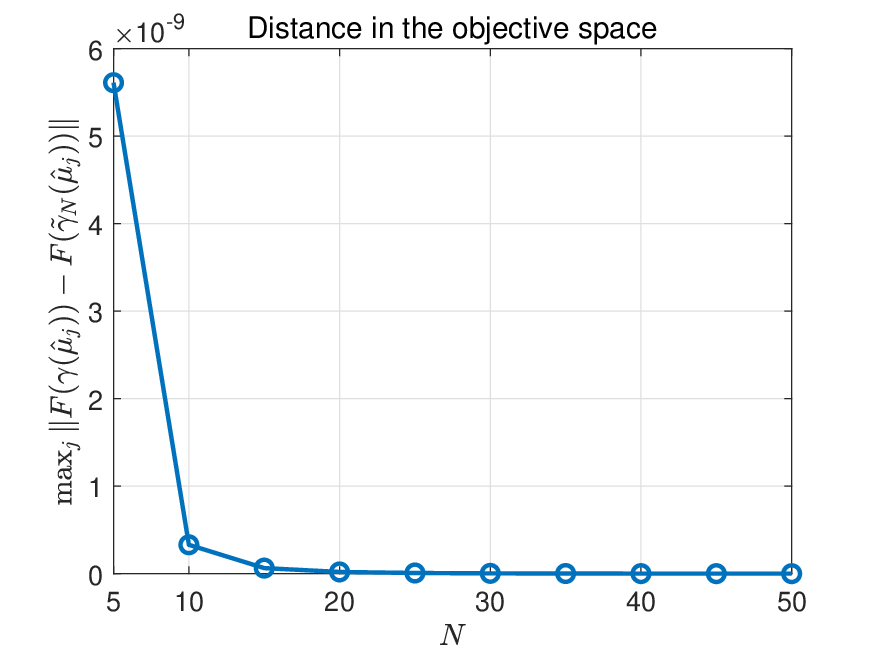}
	\label{fig:sub1.2}
\end{subfigure}

\begin{subfigure}
	\centering
	\includegraphics[width=0.48\textwidth]{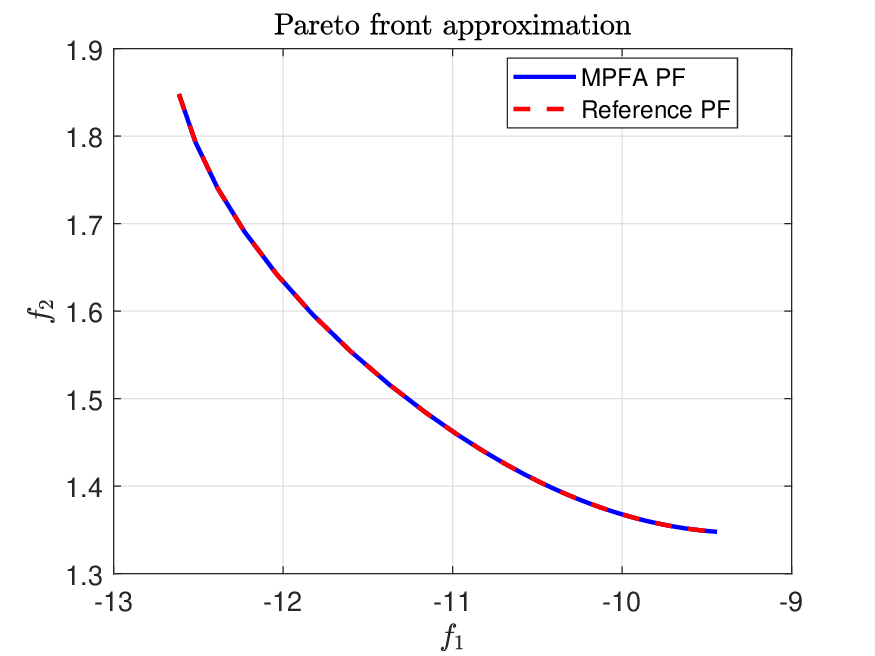}
	\label{fig:sub1}
\end{subfigure}
\caption{The results of the MPFA method for sparse PCA with $A = \Sigma U^\top \in \mathbb{R}^{3\times 6}$.}
\label{figA0d}
\end{figure}

For the following bi-objective sparse PCA problem
\begin{equation}\label{ex1}
\min_{x \in \mathbb{S}^{n-1}} \Bigl( -x^\top A^\top A x,\ \|x\|_{1, 10} \Bigr).
\end{equation}
We compute the error using the MPFA method with  $\hat N=1000$ and $\mathcal{N}=\{20, 40, 60, 80, 100,~120,~150, 180, 200\}$. Fig.~ \ref{figA0d}  illustrates the approximate error on the manifold, the induced error in the objective space, and  comparative  results of
the approximate Pareto front obtained by MPFA (MPFA PF) and 
the  reference front obtained by anchor points generated by Algorithm \ref{alg:mpfa} (Reference PF). The convergence observed on the manifold level is faithfully reflected in the objective space. Moreover, the curves of MPFA PF and Reference PF are almost indistinguishable, demonstrating that MPFA method (Algorithm \ref{alg:mpfa}) accurately captures the global structure of the Pareto front.

To further illustrate the convergence behavior of the proposed Manifold Pareto front approximation method on realistic data, we also test the algorithm with a random data matrix generated from a standard normal distribution. Let $A \in \mathbb{R}^{3 \times 3}$ be a random matrix with entries independently drawn from $\mathcal{N}(0,1)$. We solve the problem \eqref{ex1} with starting from a random initial point uniformly distributed on $\mathbb{S}^{n-1}$ by MPFA method. 
The logarithmic plot of approximate error is shown in Fig. \ref{figAd1}. The numerical results suggest that the estimated convergence order is approximately $3.5478$, which is consistent with the theoretical convergence rate predicted by Theorem \ref{herror}.   As shown in Fig.~\ref{mpfaAsx1},  the continuous interpolation curve generated by Algorithm \ref{alg:mpfa} passes through all computed efficient solutions of problem \eqref{bo}. Fig. \ref{mpfaAfrontx1} shows that the corresponding continuous Pareto front approximation passes through the objective vectors of all these efficient solutions. The MATLAB code used to generate the sparse PCA experiments in
Figures~\ref{figAd1}-\ref{mpfaAfrontx1} is publicly available at
\url{https://github.com/lkp-create/mpfa-riemannian-pareto-front}.

\begin{figure}[htbp]
\centering
\begin{subfigure}
	\centering
	\includegraphics[width=0.48\textwidth]{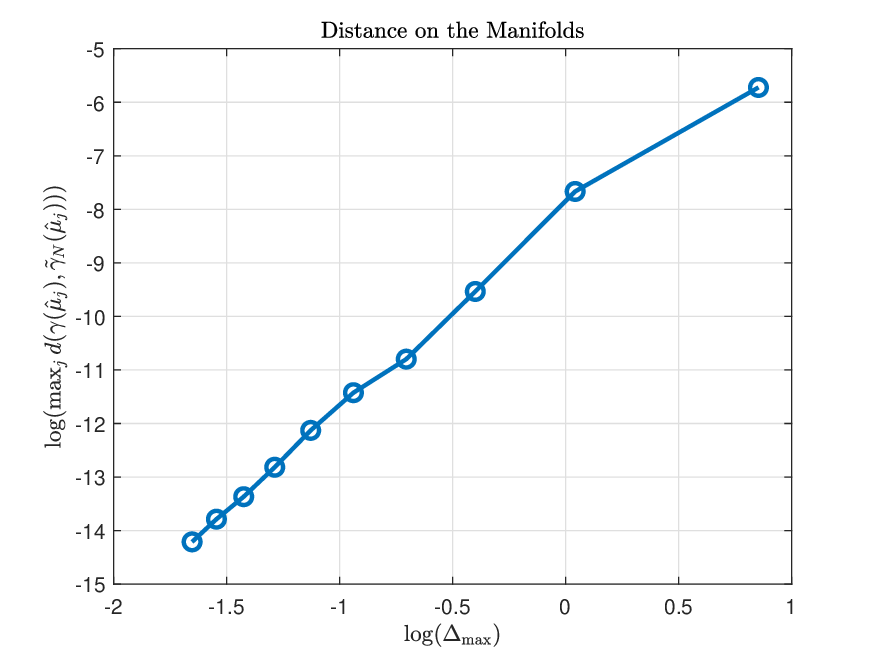}
	\label{fig:sub1}
\end{subfigure}
\hfill
\begin{subfigure}
	\centering
	\includegraphics[width=0.48\textwidth]{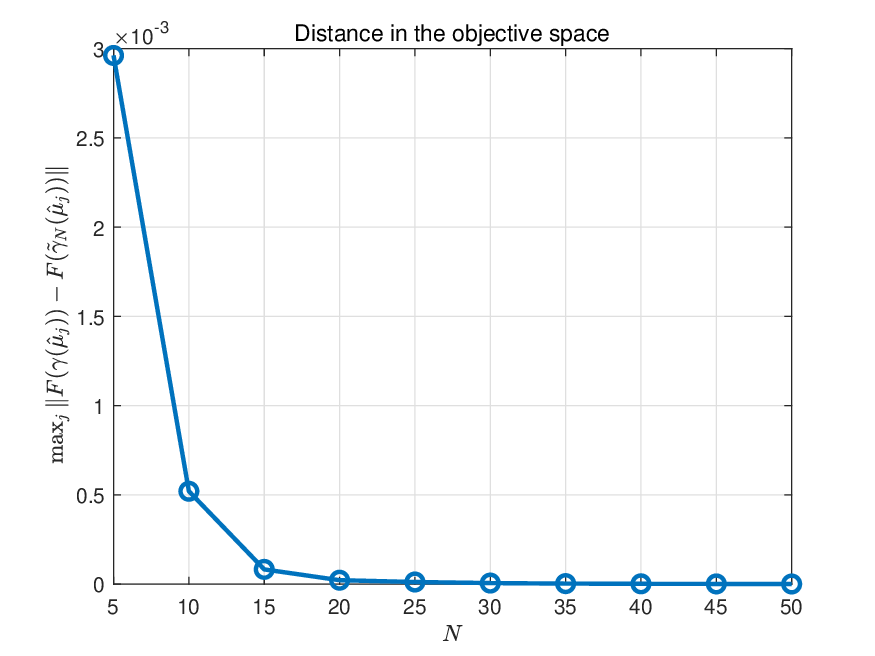}
	\label{fig:sub1.2}
\end{subfigure}
\caption{Error comparison results of the MPFA method for sparse PCA with random matrix $A$.}
\label{figAd1}
\end{figure}

\begin{figure}[htbp]
\centering
\begin{minipage}{0.48\textwidth}
	\centering
	\includegraphics[width=\linewidth]{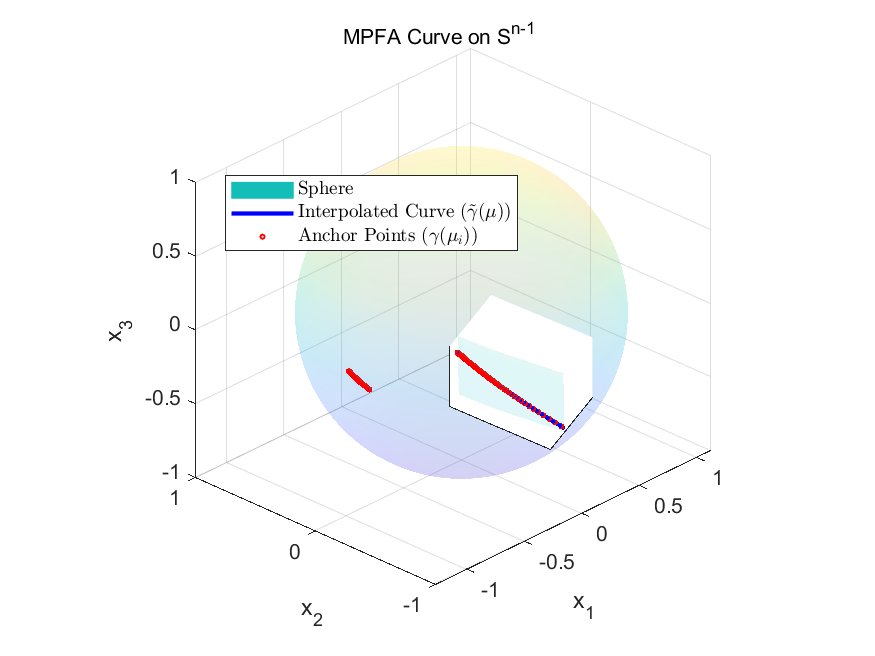}
	\captionof{figure}{Approximate solution curve.}
	\label{mpfaAsx1}
\end{minipage}
\hfill
\begin{minipage}{0.48\textwidth}
	\centering
	\includegraphics[width=\linewidth]{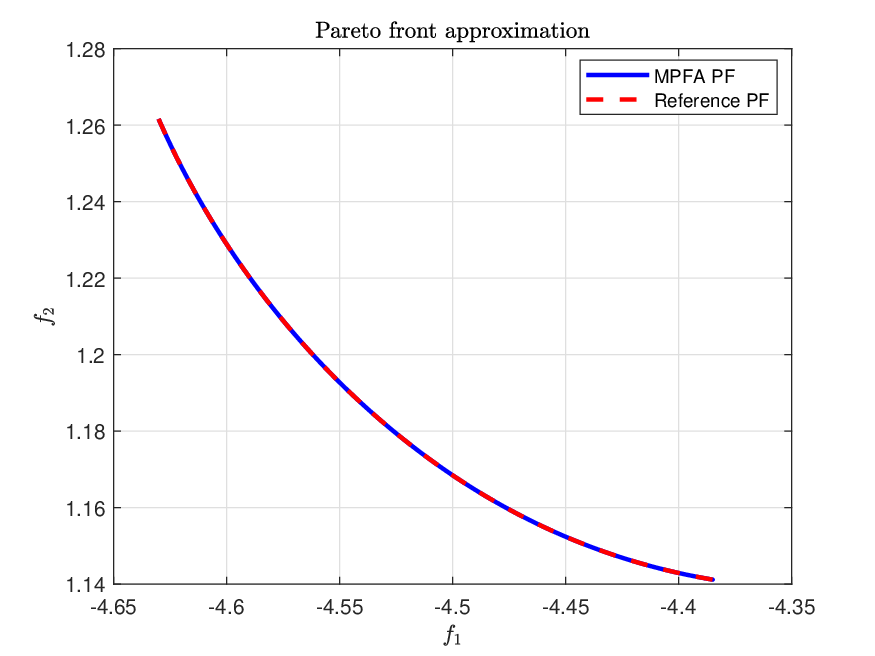}
	\captionof{figure}{Approximate Pareto front.}
	\label{mpfaAfrontx1}
\end{minipage}
\end{figure}

To further illustrate the computational efficiency of different interpolation strategies, we plot the approximation error versus the computational time.  Fig. \ref{fig:errorT} demonstrates that Hermite interpolation achieves higher efficiency, and the blended cubic spline achieves small approximation errors, but requires a longer elapsed time.

\begin{figure}[htbp]
\centering
\includegraphics[width=0.6\textwidth]{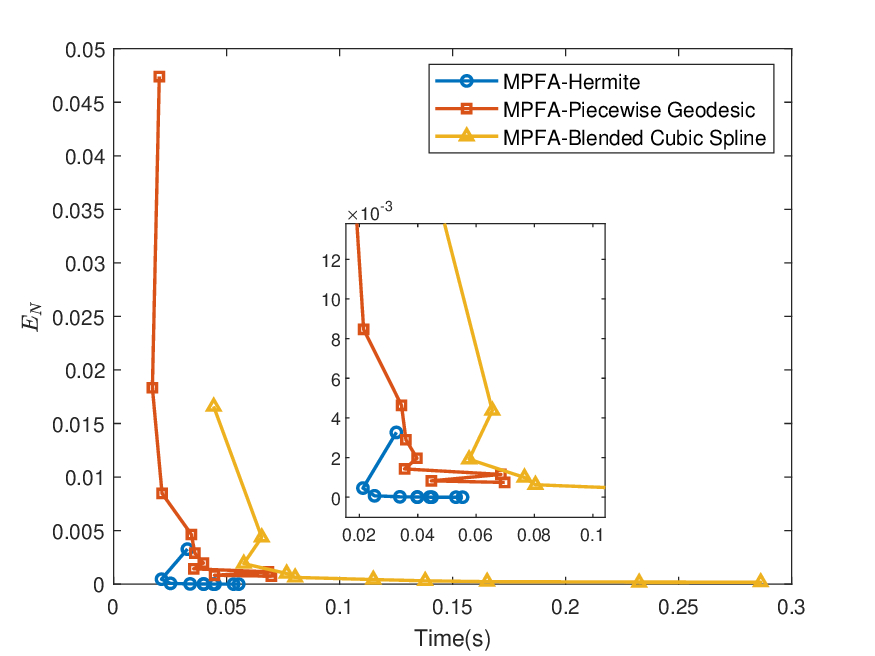}
\caption{Error comparison for different interpolation methods by Algorithm \ref{alg:evaluate} with $\mathcal N=\{5,10,15,20,25,30,35,40,45,50\}$.}
\label{fig:errorT}
\end{figure}

We can highlight the geometric differences among the three interpolation strategies by introducing another synthetic bi-objective test problem on $\mathbb S^2$:
\begin{equation}\label{eq:curved-example}
\min_{x\in\mathbb S^2}
F(x)
=
\bigl(f_1(x),f_2(x)\bigr),
\end{equation}
where
\[
f_1(x)
=
-4x_1+2x_2^2+\frac{1}{2}x_3^2
+2x_1x_2x_3,
\]
and
\[
f_2(x)
=
4x_1+2(x_2-0.7)^2
+\frac{1}{2}(x_3+0.6)^2
-2x_1x_2x_3.
\]
Fig.~\ref{fig:curved-path} shows the reference solution path $\gamma$ and the interpolated curves $\tilde \gamma$ obtained by the three interpolation strategies. Owing to the pronounced curvature of the solution path and the small number of anchor points, the differences among the methods are clear.  The Hermite interpolant demonstrates the  more accurate reproduction of the reference path.
\begin{figure}[htbp]
\centering
\includegraphics[width=1\textwidth]{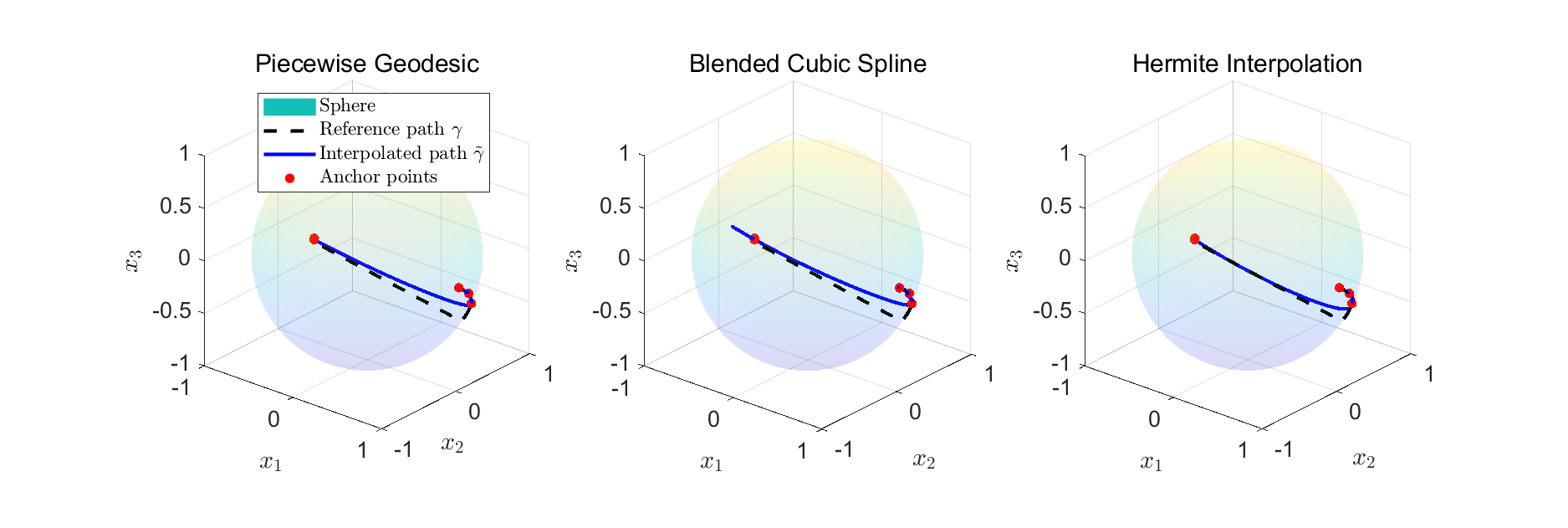}
\caption{Comparison of the reference solution path $\gamma$ and the
	interpolated curve $\tilde{\gamma}$ for the synthetic problem with $N=5$.}
\label{fig:curved-path}
\end{figure}

\section{Conclusions}\label{sec6}

In this paper, we have presented an MPFA method for bi-objective optimization problems on Riemannian manifolds based on a Hermite interpolation technique. The proposed MPFA method constructs an approximation of the Pareto front by tracing a parameterized solution curve starting from a single initial point.
The main idea is to exploit the smooth dependence of the solution of the scalarized problem on the parameter and to approximate the resulting Pareto curve using interpolation techniques on manifolds. Theoretical analysis establishes convergence properties of the proposed framework and provides insight into the approximation accuracy of the constructed Pareto front.
Numerical experiments on several test problems demonstrate that the proposed method is able to accurately capture the Pareto fronts. Comparisons among different interpolation strategies further illustrate the effectiveness of the approach. In addition, the application to bi-objective sparse principal component analysis shows that the proposed method can provide a practical tool for exploring trade-offs between competing objectives in data analysis problems.

\backmatter

\bmhead{Acknowledgements}

This work was partially supported by the National Natural Science Foundation of China (Nos. 12471299, 12071379), the Natural Science Foundation of Chongqing (No. CSTB2025NSCQ-GPX0529), the Fundamental Research Funds for the Central Universities (SWU-KF26001),
the China Scholarship Council program (No. 202506990031), and the Fonds de la Recherche Scientifique--FNRS under Grant no T.0001.23.

\section*{Declaration}
\bmhead{Conflict of interest}
 The authors declare that they have no conflict of interest.


\begin{appendices}

\section{Bi-objective test problems} \label{appendix}

Problem 1 \cite{hh06}:
\begin{equation}
	\begin{split}
		f_1(x_1, x_2) &= -(5 - \frac{(x_1^2 + x_2 - 11)^2 + (x_1 + x_2^2 - 7)^2}{200}),\\		
		f_2(x_1, x_2) &= f_1(2x_1, 2x_2);
	\end{split}
\end{equation}
Problem 2 \cite{te19}:
\begin{equation}
	\begin{split}
		f_1(x) & = \frac{1}{2}x_1^2 + x_2^2 - 10x_1 - 100, \\
		f_2(x) & = x_1^2 + \frac{1}{2}x_2^2 - 10x_2 - 100;
	\end{split}
\end{equation}
Problem 3 \cite{te19}:
\begin{equation}
	\begin{split}
		f_1(x) & = \sin x_2, \\
		f_2(x) & = 1 - \exp\left(-\left(x_1 - \frac{1}{\sqrt{2}}\right)^2 - \left(x_2 - \frac{1}{\sqrt{2}}\right)^2\right);
	\end{split}
\end{equation}
Problem 4 \cite{sk02}:
\begin{equation}
	\begin{split}
		f_1(x)
		&=(x_1-2)^2+(x_2+3)^2+(x_3-5)^2+(x_4-4)^2-5,\\
		f_2(x)
		&=-\frac{\sin(x_1)+\sin(x_2)+\sin(x_3)+\sin(x_4)}
		{1+{x_1^2+x_2^2+x_3^2+x_4^2}/{100}};	
	\end{split}
\end{equation}
Problem 5 \cite{hh06}:
\begin{equation}
	\begin{split}
		f_1(x_1, x_2) & = (x_1 - 1)^2 + (x_1 - x_2)^2, \\
		f_2(x_1, x_2) & = (x_2 - 3)^2 + (x_1 - x_2)^2;
	\end{split}
\end{equation}
Problem 6 \cite{pn06}:
\begin{equation}
	\begin{split}
		f_1(x_1, x_2) & = x_1^4 + x_2^4 - x_1^2 + x_2^2 + 10x_1x_2 + \frac{1}{4}x_1 + 20, \\
		f_2(x_1, x_2) & = (x_1 - 1)^2 + x_2^2.
	\end{split}
\end{equation}

\end{appendices}



\end{document}